\documentclass[a4paper,fleqn,longmktitle]{cas-sc}

\usepackage[numbers]{natbib}
\usepackage{amssymb}
\usepackage{lineno,hyperref}
\usepackage{graphicx}
\usepackage{subfig}
\usepackage{tikz,nicefrac,pifont,fp}
\usetikzlibrary{decorations.markings}
\usetikzlibrary{calc,positioning,decorations.pathmorphing,decorations.pathreplacing,fixedpointarithmetic}
\usepackage{figs/figstyle}

\def\tsc#1{\csdef{#1}{\textsc{\lowercase{#1}}\xspace}}
\tsc{WGM}
\tsc{QE}
\tsc{EP}
\tsc{PMS}
\tsc{BEC}
\tsc{DE}

\newcommand{\sC}{\mathcal{C}}
\newcommand{\sP}{\mathcal{P}}
\newcommand{\K}{{\mathcal{K}}}

\newtheorem{theorem}{Theorem}
\newtheorem{lemma}[theorem]{Lemma}
\newproof{proof}{Proof}

\begin{document}
\let\WriteBookmarks\relax
\def\floatpagepagefraction{1}
\def\textpagefraction{.001}
\shorttitle{ }
\shortauthors{Freitas and Lee}

\title [mode = title]{Some results on structure of all arc-locally (out) in-semicomplete digraphs}

\author[1]{Lucas Ismaily Bezerra Freitas}[type=editor,
                        auid=000,bioid=1,]
\ead{ismailybf@ic.unicamp.br}

\credit{Conceptualization of this study, Methodology, Software}

\address[1]{Institute of Computing, State University of Campinas, Campinas, Brazil}

\author[1]{Orlando Lee}[orcid=0000-0003-4462-3325]
\fnmark[1]
\ead{lee@ic.unicamp.br}

\credit{Data curation, Writing - Original draft preparation}

\cortext[cor2]{Principal corresponding author}
\fntext[fn2]{This author was supported by CNPq Proc. 303766/2018-2, CNPq Proc 425340/2016-3 and FAPESP Proc. 2015/11937-9.}

\begin{abstract}
Arc-locally semicomplete and arc-locally in-semicomplete digraphs were introduced by Bang-Jensen as a common generalization of both semicomplete and semicomplete bipartite digraphs in 1993. Later, Bang-Jensen (2004), Galeana-Sánchez and Goldfeder (2009) and Wang and Wang (2009) provided a characterization of strong arc-locally semicomplete digraphs. In 2009, Wang and Wang characterized strong arc-locally in-semicomplete digraphs. In 2012, Galeana-Sánchez and Goldfeder provided a characterization of all arc-locally semicomplete digraphs which generalizes some results by Bang-Jensen. In this paper, we characterize the structure of arbitrary connected arc-locally (out) in-semicomplete digraphs and arbitrary connected arc-locally semicomplete digraphs. 
\end{abstract}




\maketitle

\section{Introduction}
\label{intro}

We only consider finite digraphs without loops and multiple edges. In this section we give an overview of the literature related to the results we prove in this paper. We assume the reader is familiarized with graph theory terminology and we postpone formal definitions to Section~\ref{notation}. 

Some very important results in graph theory characterize a certain class of graphs (or digraphs) in terms of certain forbidden induced subgraphs (subdigraphs). In~\cite{alis}, Bang-Jensen introduced interesting classes of digraphs in terms of certain forbidden induced subgraphs that generalize both semicomplete and semicomplete bipartite digraphs. There are four different possible orientations of the $P_4$, see Figure~\ref{P4-orientacao}. Consider the digraphs $H_1$, $H_2$, $H_3$ and $H_4$ of Figure~\ref{P4-orientacao}. For $i \in \{1, 2, 3, 4\}$, we say that $D$ is an \emph{orientedly $\{H_i\}$-free digraph} if for all $H_i$ in $D$, the vertices $v_1$ and $v_4$ of $H_i$ are adjacent. The orientedly $\{H_{1}\}$-free digraphs (resp.,  orientedly $\{H_{2}\}$-free digraphs) are called \emph{arc-locally in-semicomplete digraphs} (resp., \emph{arc-locally out-semicomplete digraphs}), orientedly $\{H_{3}\}$-free digraphs are called \emph{3-quasi-transitive digraphs}, orientedly $\{H_{4}\}$-free digraphs are called \emph{3-anti-quasi-transitive digraphs} and orientedly $\{H_{1},H_{2}\}$-free digraphs are called \emph{arc-locally semicomplete digraphs}.

\begin{figure}[htbp]
\begin{center}
\subfloat[$H_1$]{
       \tikzset{middlearrow/.style={
	decoration={markings,
		mark= at position 0.6 with {\arrow{#1}},
	},
	postaction={decorate}
}}

\tikzset{shortdigon/.style={
	decoration={markings,
		mark= at position 0.45 with {\arrow[inner sep=10pt]{<}},
		mark= at position 0.75 with {\arrow[inner sep=10pt]{>}},
	},
	postaction={decorate}
}}

\tikzset{digon/.style={
	decoration={markings,
		mark= at position 0.4 with {\arrow[inner sep=10pt]{<}},
		mark= at position 0.6 with {\arrow[inner sep=10pt]{>}},
	},
	postaction={decorate}
}}

\begin{tikzpicture}[scale = 0.45]		
	\node (n3) [black vertex] at (9,5)  {};
	\node (n4) [black vertex] at (9,8)  {};
	\node (n1) [black vertex] at (4,8)  {};
	\node (n2) [black vertex] at (4,5)  {};
	
	\node (label_n3)  at (9.5,4.5)  {$v_3$};
	\node (label_n4)  at (9.5,8.5)  {$v_4$};
	\node (label_n1)  at (3.5,8.5)  {$v_1$};
	\node (label_n2)  at (3.5,4.5)  {$v_2$};

  \foreach \from/\to in {n1/n2,n2/n3,n4/n3}
    \draw[edge,middlearrow={>}] (\from) -- (\to);    
\end{tikzpicture}
}
\quad
\subfloat[$H_2$]{
       \tikzset{middlearrow/.style={
	decoration={markings,
		mark= at position 0.6 with {\arrow{#1}},
	},
	postaction={decorate}
}}

\tikzset{shortdigon/.style={
	decoration={markings,
		mark= at position 0.45 with {\arrow[inner sep=10pt]{<}},
		mark= at position 0.75 with {\arrow[inner sep=10pt]{>}},
	},
	postaction={decorate}
}}

\tikzset{digon/.style={
	decoration={markings,
		mark= at position 0.4 with {\arrow[inner sep=10pt]{<}},
		mark= at position 0.6 with {\arrow[inner sep=10pt]{>}},
	},
	postaction={decorate}
}}

\begin{tikzpicture}[scale = 0.45]		
	\node (n3) [black vertex] at (9,5)  {};
	\node (n4) [black vertex] at (9,8)  {};
	\node (n1) [black vertex] at (4,8)  {};
	\node (n2) [black vertex] at (4,5)  {};
	
	\node (label_n3)  at (9.5,4.5)  {$v_3$};
	\node (label_n4)  at (9.5,8.5)  {$v_4$};
	\node (label_n1)  at (3.5,8.5)  {$v_1$};
	\node (label_n2)  at (3.5,4.5)  {$v_2$};

  \foreach \from/\to in {n2/n1,n2/n3,n3/n4}
    \draw[edge,middlearrow={>}] (\from) -- (\to);    
\end{tikzpicture}
}

\quad
\subfloat[$H_3$]{
       \tikzset{middlearrow/.style={
	decoration={markings,
		mark= at position 0.6 with {\arrow{#1}},
	},
	postaction={decorate}
}}

\tikzset{shortdigon/.style={
	decoration={markings,
		mark= at position 0.45 with {\arrow[inner sep=10pt]{<}},
		mark= at position 0.75 with {\arrow[inner sep=10pt]{>}},
	},
	postaction={decorate}
}}

\tikzset{digon/.style={
	decoration={markings,
		mark= at position 0.4 with {\arrow[inner sep=10pt]{<}},
		mark= at position 0.6 with {\arrow[inner sep=10pt]{>}},
	},
	postaction={decorate}
}}

\begin{tikzpicture}[scale = 0.45]		
	\node (n3) [black vertex] at (9,5)  {};
	\node (n4) [black vertex] at (9,8)  {};
	\node (n1) [black vertex] at (4,8)  {};
	\node (n2) [black vertex] at (4,5)  {};
	
	\node (label_n3)  at (9.5,4.5)  {$v_3$};
	\node (label_n4)  at (9.5,8.5)  {$v_4$};
	\node (label_n1)  at (3.5,8.5)  {$v_1$};
	\node (label_n2)  at (3.5,4.5)  {$v_2$};

  \foreach \from/\to in {n1/n2,n2/n3,n3/n4}
    \draw[edge,middlearrow={>}] (\from) -- (\to);    
\end{tikzpicture}
}
\quad
\subfloat[$H_4$]{
       \tikzset{middlearrow/.style={
	decoration={markings,
		mark= at position 0.6 with {\arrow{#1}},
	},
	postaction={decorate}
}}

\tikzset{shortdigon/.style={
	decoration={markings,
		mark= at position 0.45 with {\arrow[inner sep=10pt]{<}},
		mark= at position 0.75 with {\arrow[inner sep=10pt]{>}},
	},
	postaction={decorate}
}}

\tikzset{digon/.style={
	decoration={markings,
		mark= at position 0.4 with {\arrow[inner sep=10pt]{<}},
		mark= at position 0.6 with {\arrow[inner sep=10pt]{>}},
	},
	postaction={decorate}
}}

\begin{tikzpicture}[scale = 0.45]		
	\node (n3) [black vertex] at (9,5)  {};
	\node (n4) [black vertex] at (9,8)  {};
	\node (n1) [black vertex] at (4,8)  {};
	\node (n2) [black vertex] at (4,5)  {};
	
	\node (label_n3)  at (9.5,4.5)  {$v_3$};
	\node (label_n4)  at (9.5,8.5)  {$v_4$};
	\node (label_n1)  at (3.5,8.5)  {$v_1$};
	\node (label_n2)  at (3.5,4.5)  {$v_2$};

  \foreach \from/\to in {n2/n1,n2/n3,n4/n3}
    \draw[edge,middlearrow={>}] (\from) -- (\to);    
\end{tikzpicture}
}
\caption{\centering Orientations of the $P_4$.}
\label{P4-orientacao}
\end{center}
\end{figure}
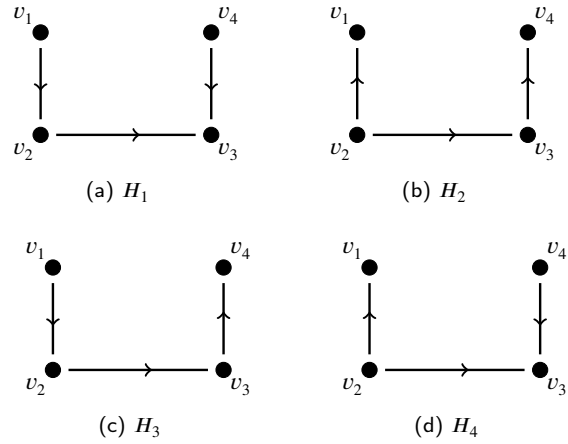

Several structural results for the classes defined above are known. In~\cite{bang2004}, Bang-Jensen provided a characterization for strong arc-locally semicomplete digraphs, but Galeana-Sánchez and Goldfeder in~\cite{galeana2009} and Wang and Wang in~\cite{bang2004} independently pointed out that one family of strong arc-locally semicomplete digraphs was missing. In~\cite{wang2009structure}, Wang and Wang characterized strong arc-locally in-semicomplete digraphs. In~\cite{Galeana-Sanchez2010}, Galeana-Sánchez, Goldfeder and Urrutia characterized strong 3-quasi-transitive digraphs. To the best of our knowledge, no characterization for the class of 3-anti-quasi-transitive digraphs is known. In~\cite{wang2014}, Wang defined a subclass of 3-anti-quasi-transitive digraphs. A digraph $D$ is a \emph{3-anti-circulant digraph} if for any four distinct vertices $x_1, x_2, x_3, x_4 \in V(D)$, if $x_1x_2$, $x_3x_2$ and $x_3x_4$ are edges in $E(D)$, then $x_4x_1$ is in $E(D)$. Wang~\cite{wang2014} characterized the structure of 3-anti-circulant digraphs containing a cycle factor and showed that the structure is very close to semicomplete and semicomplete bipartite digraphs. In~\cite{galeana2012}, Galeana-Sánchez and Goldfeder extended the Bang-Jensen results in \cite{bang2004} and characterized all digraphs arc-locally semicomplete, this is the only class among the ones defined previously that has a characterization for arbitrary digraphs. In this paper, we characterize the structure of arbitrary arc-locally (out) in-semicomplete digraphs and arbitrary arc-locally semicomplete digraphs. We show that the structure of these digraphs is very similar to diperfect digraphs. 

The rest of this paper is organized as follows. In Section~\ref{notation}, we present the basic concepts of digraphs and the notation used. In Section~\ref{arc-in-semi}, we show that if $D$ is a connected arc-locally (out) in-semicomplete digraph, then $D$ is diperfect, or $D$ admits a special partition of its vertices, or $D$ has a clique cut. In Section~\ref{arc-semi}, we show that if $D$ is a connected arc-locally semicomplete digraph, then $D$ is either a diperfect digraph or an odd extended cycle of length at least five.

\section{Notation}
\label{notation}

We consider that the reader is familiar with the basic concepts of graph theory. Thus, for details that are not present in this paper, we refer the reader to~\cite{bang2008digraphs,Bondy08}.

Let $D$ be a digraph with vertex set $V(D)$ and edge set $E(D)$. Given two vertices $u$ and $v$ of $V(D)$, we say that $u$ \emph{dominates} $v$, denoted by $u \to v$, if $uv \in E(D)$. We say that $u$ and $v$ are \emph{adjacent} if $u \to v$ or $v \to u$. A digraph $H$ is a \emph{subdigraph} of $D$ if $V(H)\subseteq V(D)$ and $E(H) \subseteq E(D)$. If every edge of $E(D)$ with both vertices in $V(H)$ is in $E(H)$, then we say that $H$ is \emph{induced} by $X = V(H)$, we write $H = D[X]$. We say that $H$ is an induced subdigraph of $D$ if there is $X \subseteq V(D)$ such that $H = D[X]$. If every pair of distinct vertices of $D$ are adjacent, we say that $D$ is a \emph{semicomplete digraph}. The \emph{underlying graph} of a digraph $D$, denoted by $U(D)$, is the simple graph defined by $V(U(D))= V(D)$ and $E(U(D))= \{uv : u $ and $v$ are adjacent in $D\}$. Whenever it is appropriate, we may borrow terminology from undirected graphs to digraphs. For instance, we may that a digraph $D$ is connected if $U(D)$ is connected. The \emph{inverse digraph} of $D$ is the digraph with vertex set $V(D)$ and edge set $\{uv : vu \in E(D)\}$.

A \emph{path} $P$ in a digraph $D$ is a sequence of distinct vertices $P = v_1v_2 \dots v_k$, such that for all $v_i \in V(P)$, $v_iv_{i+1} \in E(D)$, with $1 \leq i \leq k-1$. We say that $P$ is a path that \emph{starts} at $v_1$ and \emph{ends} at $v_k$. We define the \emph{length} of $P$ as $k-1$. We denote by $P_k$ the class of isomorphism of a path of length $k-1$. For disjoint subsets $X$ and $Y$ of $V(D)$ (or subdigraphs of $D$), we say that $X$ \emph{reaches} $Y$ if there are $u \in X$ and $v \in Y$ such that there exists a path from $u$ to $v$ in $D$. The \emph{distance} between two vertices $u,v \in V(D)$, denoted by $dist(u,v)$, is the length of the shortest path from $u$ to $v$. We can extend the concept of distance to subsets of $V(D)$ or subdigraphs of $D$, that is, the distance from $X$ to $Y$ is $dist(X,Y)= \min\{dist(u,v): u\in X$ and $v \in Y \}$. A \emph{cycle} $C$ is a sequence of vertices $C = v_1v_2 \dots v_kv_1$, such that $v_1v_2 \dots v_k$ is a path, $v_kv_1 \in E(D)$ and $k>1$. We define the \emph{length} of $C$ as k. If $k$ is odd, then we say that $C$ is an \emph{odd cycle}. We say that $D$ is an \emph{acyclic digraph} if $D$ does not contain cycles. We say that $C$ is a \emph{non-oriented cycle} if $C$ is not a cycle in $D$, but $U(C)$ is a cycle in $U(D)$. In particular, if a non-oriented cycle $C$ has length three, then we say that $C$ is a \emph{transitive triangle}. 

Let $G$ be an undirected graph. A \emph{clique} in $G$ is a subset $X$ of $V(G)$ such that $G[X]$ is complete. The \emph{clique number} of $G$, denoted by $\omega(G)$, is the size of maximum clique of $G$. A subset $S$ of $V(G)$ is \emph{stable} if every pair of vertices in $S$ are pairwise non-adjacent. A \emph{(proper) coloring} of $G$ is a partition of $V(G)$ into stable sets $\{S_1,\ldots, S_k\}$. The \emph{chromatic number} of $G$, denoted by $\chi(G)$, is the cardinality of a minimum coloring of $G$. A graph $G$ is \emph{perfect} if for every induced subgraph $H$ of $G$, equality $\omega(H)=\chi(H)$ holds. We say that a digraph $D$ is \emph{diperfect} if $U(D)$ is perfect.

Let $D$ be a connected digraph. We say that $D$ is \emph{strong} if for each pair of vertices $u,v \in V(D)$, then there exists a path from $u$ to $v$ in $D$. A \emph{strong component} of $D$ is a maximal induced subgraph of $D$ which is strong. Let $Q$ be a strong component of $D$. We define $\K^-(Q)$ (resp., $\K^+(Q)$) as the set of strong components that reach (resp., are reached by) $Q$. We say that $Q$ is an \emph{initial strong component} if there exists no vertex $v$ in $D-V (Q)$ such that $v$ dominates some vertex of $Q$. The vertex set $B \subset V(D)$ is a \emph{vertex cut} if $D-B$ is a disconnected digraph. If $D[B]$ is a semicomplete digraph, then we say that $B$ is a \emph{clique cut}. For disjoint subsets $X$ and $Y$ of $V(D)$ (or subdigraphs of $D$), we say that $X$ and $Y$ are \emph{adjacent} if some vertex of $X$ and some vertex of $Y$ are adjacent; $X \to Y$ means that every vertex of $X$ dominates every vertex of $Y$, $X \Rightarrow Y$ means that there is no edge from $Y$ to $X$ and $X \mapsto Y$ means that both of $X \to Y$ and $X \Rightarrow Y$ hold. When $X = \{x\}$ or $Y = \{y\}$, we simply write the element, as in $x \mapsto Y$ and $X \mapsto y$.

\section{Arc-locally in-semicomplete digraphs}
\label{arc-in-semi}

Let us start with a class of digraphs which is related to arc-locally in-semicomplete digraphs. Let $C$ be a cycle of length $k \geq 2$ and let $X_1 , X_2 , \ldots , X_k$ be disjoint stable sets. The \emph{extended cycle} $C:= C[ X_1 , X_2 , \ldots , X_k]$ is the digraph with vertex set $X_1 \cup X_2 \cup \cdots \cup X_k$ and edge set $\{x_ix_{i+1} : x_i \in X_i, x_{i+1} \in X_{i + 1}, i=1,2,\ldots,k \}$, where subscripts are taken modulo $k$. So  $X_1 \mapsto X_2 \mapsto \cdots \mapsto X_k \mapsto X_1$ (see Figure~\ref{circ-estentido}).

\begin{figure}[ht]
\begin{center}
    \tikzset{middlearrow/.style={
	decoration={markings,
		mark= at position 0.6 with {\arrow{#1}},
	},
	postaction={decorate}
}}

\tikzset{shortdigon/.style={
	decoration={markings,
		mark= at position 0.45 with {\arrow[inner sep=10pt]{<}},
		mark= at position 0.75 with {\arrow[inner sep=10pt]{>}},
	},
	postaction={decorate}
}}

\tikzset{digon/.style={
	decoration={markings,
		mark= at position 0.4 with {\arrow[inner sep=10pt]{<}},
		mark= at position 0.6 with {\arrow[inner sep=10pt]{>}},
	},
	postaction={decorate}
}}

\def \n {5}
\def \radius {3cm}
\def \margin {8}

\begin{tikzpicture}[scale = 0.66]

    \draw (4,5) circle (30pt);
	\node (v1_1) [black vertex] at (3.8,4.8)  {};
	\node (v1_2) [black vertex] at (4.2,5.3)  {};

	\draw (9,5) circle (30pt); 
	\node (v2_1) [black vertex] at (9,5)  {};

	\draw (9,8) circle (30pt);
	\node (v3_1) [black vertex] at (9.5,8)  {};
	\node (v3_2) [black vertex] at (8.5,8)  {};
	\node (v3_3) [black vertex] at (9,8)  {};

	\draw (6.5,10) circle (30pt);
	\node (v4_1) [black vertex] at (6.5,10.5) {};
	\node (v4_2) [black vertex] at (6.5,9.5) {};

    \draw (4,8) circle (30pt);	
	\node (v5_1) [black vertex] at (4,8)  {};

	\node (label_n4)  at (6.5,11.5) {$X_4$};
	\node (label_n2)  at (10.,4.0)  {$X_2$};
	\node (label_n3)  at (10.0,9.0)  {$X_3$};
	\node (label_n5)  at (3.0,9.0)  {$X_5$};
	\node (label_n1)  at (3.0,4.0)  {$X_1$};
	

  \foreach \from/\to in {v1_1/v2_1,v1_2/v2_1,v2_1/v3_1,v2_1/v3_2,v2_1/v3_3,v3_1/v4_1,v3_2/v4_1,v3_3/v4_1,v3_1/v4_2,v3_2/v4_2,v3_3/v4_2,v4_1/v5_1,v4_2/v5_1,v5_1/v1_1,v5_1/v1_2}
    \draw[edge,middlearrow={>}] (\from) -- (\to);    
\end{tikzpicture}
\caption{\centering Example of an extended cycle.}
\label{circ-estentido}
\end{center}
\end{figure}
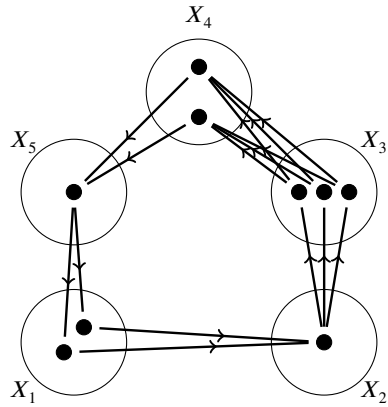

In~\cite{wang2009structure}, Wang and Wang characterized the structure of strong arc-locally in-semicomplete digraphs. We have omitted the definition of a $T$-digraph, because it is a family of digraphs that does not play an important role in this paper. For ease of reference, we state the following result.

\begin{theorem}[Wang and Wang~\cite{wang2009structure}, 2009]
\label{arc-in}
Let $D$ be a strong arc-locally in-semicomplete digraph, then $D$ is either a semicomplete digraph, a semicomplete bipartite digraph, an extended cycle or a $T$-digraph.
\end{theorem}

We start with the proof of a simple and useful lemma.

\begin{lemma}
\label{arc-in-circuito}
If $D$ is an arc-locally in-semicomplete digraph, then $D$ contains no induced non-oriented odd cycle of length at least five. 
\end{lemma}

\begin{proof}
Assume that $D$ contains an induced non-oriented odd cycle of length at least five. Let $P=u_1 u_2 \dots u_k$ be a maximum path in $C$. Note that $P$ has at least three vertices, since $C$ is odd and has at least five vertices. Let $w$ be the vertex of $C$ distinct from $u_{k-1}$ that dominates $u_k$. Since $D$ is arc-locally in-semicomplete, the vertices $w$ and $u_ {k-2} $ must be adjacent, a contradiction.
\qed
\end{proof}

The next lemma states that not containing an induced odd cycle of length at least five is a necessary and sufficient condition for an arc-locally in-semicomplete digraph to be diperfect. Note that if a digraph $D$ contains no induced odd cycle of length at least five, then $D$ also contains no induced odd extended cycle of the same length. For the next lemma, we need the famous Strong Perfect Graph Theorem~\cite{chudnovsky2006strong}.

\begin{theorem}[Chudnovsky, Robertson, Seymour and Thomas~\cite{chudnovsky2006strong}, 2006]
\label{perf}
A graph $G$ is perfect if, and only if, $G$ does not contain an induced odd cycle of length at least five or its complement as an induced subgraph. 
\end{theorem}

\begin{lemma}
\label{arc-no-cc-perf}
Let $D$ be an arc-locally in-semicomplete digraph. Then, $D$ is diperfect if and only if $D$ contains no induced odd cycle of length at least five.
\end{lemma}

\begin{proof}

Let $D$ be an arc-locally in-semicomplete digraph. By Lemma~\ref{arc-in-circuito}, the digraph $D$ contains no induced non-oriented odd cycle of length at least five. Thus, it follows that every induced odd cycle of length at least five in $U(D)$ is, also, a induced odd cycle in $D$. 

First, if $D$ is diperfect, then the result follows by Theorem~\ref{perf}. To prove sufficiency, assume that $D$ is not diperfect. Since $D$ contains no induced odd cycle of length at least five, by Theorem~\ref{perf}, the graph $U(D)$ contains an induced complement, denoted by $\overline{U(C)}$, of an odd cycle $U(C)$ of length at least five. By definition of complement, two vertices are adjacent in $\overline{U(C)}$ (and in $D$) if, and only if, they are not consecutive in $U(C)$. Since the complement of a $C_5$ is a $C_5$, we may assume that $C$ contains at least seven vertices. Let $v_{1},v_{2},v_{3},v_{4},v_{5},v_{6},v_{7}$ be consecutive vertices in $C$. Recall that since $D$ is arc-locally in-semicomplete, if $x,y,u$ and $v$ are distinct vertices such that $u$ and $v$ are adjacent, $x\to u$ and $y\to v$, then $x$ and $y$ must be adjacent in $D$. 

The remainder of the proof is divided into two cases, depending on whether $v_2 \to v_4$ or $v_4 \to v_2$.

\textbf{Case 1.} Assume that $v_2 \to v_4$. If $v_1 \to v_6$, then since $v_4$ and $v_6$ are adjacent, it follows that $v_1$ and $v_2$ are adjacent, a contradiction. So, we can assume that $v_6 \to v_1$. If $v_3\to v_7$, then since $v_2\to v_4$ and, $v_4$ and $v_7$ are adjacent, it follows that $v_2$ and $v_3$ are adjacent in $D$, a contradiction. So $v_7\to v_3$. Finally, since $v_7 \to v_3$, $v_6 \to v_1$ and, $v_1$ and $v_3$ are adjacent, it follows that $v_6$ and $v_7$ are adjacent in $D$, a contradiction.

\textbf{Case 2.} Assume that $v_4 \to v_2$. If $v_3 \to v_6$, then since $v_2$ and $v_6$ are adjacent, it follows that $v_3$ and $v_4$ are adjacent, a contradiction. So, we have $v_6 \to v_3$. If $v_7 \to v_5$, since $v_6 \to v_3$ and, $v_5$ and $v_3$ are adjacent, it follows that $v_7$ and $v_6$ are adjacent in $D$, a contradiction. So, we have $v_5 \to v_7$. Finally, since $v_5 \to v_7$, $v_4 \to v_2$ and, $v_2$ and $v_7$ are adjacent, it follows that $v_5$ and $v_4$ are adjacent in $D$, a contradiction.

Thus $D$ is diperfect.
\qed
\end{proof}

Next, we prove some properties of an arc-locally in-semicomplete digraph $D$, when $D$ has a strong component that induces an odd extended cycle of length at least five. To do this, we will use the following auxiliary results. 

\begin{lemma}[Wang e Wang~\cite{wang2011independent}, 2011]
\label{arc-in-lema_2}
Let $D$ be an arc-locally in-semicomplete digraph and let $H$ be a non-trivial strong subdigraph of $D$. For any $v \in V(D)-V(H)$, if there exists a path from $v$ to $H$, then $v$ and $H$ are adjacent. In particular, if $H$ is a strong component, then $v$ dominates some vertex of $H$.
\end{lemma}

\begin{lemma}[Wang e Wang~\cite{wang2011independent}, 2011]
\label{arc-in-lema_4}
Let $D$ be an arc-locally in-semicomplete digraph and let $K_1$ and $K_2$ be two distinct non-trivial strong components of $D$ with at least one edge from $K_1$ to $K_2$. Then either $K_1 \mapsto K_2$ or $K_1 \cup K_2$ is a bipartite digraph.
\end{lemma}

\begin{lemma}[Wang and Wang~\cite{wang2011independent}, 2011]
\label{arc-in-lema_1}
Let $D$ be an arc-locally in-semicomplete digraph and let $Q$ be a non-trivial strong component of $D$. Let $v$ be a vertex of $V(D)-V(Q)$ that dominates some vertex of $Q$. If $D[V(Q)]$ is non-bipartite, then $v \mapsto Q$.
\end{lemma}

Recall that, if $Q$ is a strong component of a digraph $D$, then $\K^-(Q)$ (resp., $\K^+(Q)$) is the set of strong components that reach (resp., are reached by) $Q$ in $D$.

\begin{lemma}
\label{lem-arc-cir-est}
Let $D$ be a non-strong arc-locally in-semicomplete digraph. Let $Q$ be a non-initial strong component of $D$ that induces an odd extended cycle of length at least five. Let $W = \cup_{K \in \K^-(Q)}V(K)$. Then, each of the following holds: 

\begin{enumerate}
\item[\rm{i}.] every strong component of $\K^+(Q)$ is trivial;
\item[\rm{ii}.] $W \mapsto Q$; 
\item[\rm{iii}.] $D[W]$ is a semicomplete digraph;
\item[\rm{iv}.] there exists a unique initial strong component that reaches $Q$ in $D$.
\end{enumerate}

\end{lemma}

\begin{proof}
Let $Q:=Q[X_1,X_2,...,X_{2k+1}]$. be a non-initial strong component that induces an odd extended cycle of length at least five of $D$.

\rm{i}. Towards a contradiction, assume that there exists a non-trivial strong component $K$ in $\K^+(Q)$. By the definition of $\K^+(Q)$, there exists a path from $Q$ to $K$. By Lemma~\ref{arc-in-lema_2}, there must be some edge from $Q$ to $K$. Note that $D[V(Q)]$ is a non-bipartite digraph. Thus, it follows by Lemma~\ref{arc-in-lema_4} that $Q \mapsto K$. Let $uv$ be an edge of $K$ and, let $x_1 \in X_1$ and $x_3 \in X_3$ be vertices of $Q$. Since $x_1 \to u$, $x_3 \to v$ and $D$ is arc-locally in-semicomplete, then $x_1$ and $x_3$ are adjacent, a contradiction to the fact that $Q$ induces an extended cycle. Therefore, every strong component of $\K^+(Q)$ is trivial.

\rm{ii}. Let $v$ be a vertex of $W$. By the definition of $\K^-(Q)$ and $W$, there exists a path from $v$ to $Q$. By Lemma~\ref{arc-in-lema_2}, the vertex $v$ dominates some vertex of $Q$. Since $D[V(Q)]$ is a non-bipartite digraph, it follows by Lemma~\ref{arc-in-lema_1} that $v \mapsto Q$. 

\rm{iii}. Towards a contradiction, assume that there are two non-adjacent vertices $u$ and $v$ in $W$. By (ii), we have that $\{u,v\} \mapsto Q$. Let $xy$ be an edge of $Q$. Since $D$ is arc-locally in-semicomplete, $u \to x$ and $v \to y$, it follows that $u$ and $v$ are adjacent, a contradiction. So, all vertices in $W$ are adjacent, and therefore, the vertex set $W$ induces a semicomplete digraph.

\rm{iv}. Towards a contradiction, assume that $D$ has two initial strong components, say $K_1$ and $K_2$, that reach $Q$. By (iii), $D[W]$ is a semicomplete digraph. Since $\{V(K_1) \cup V(K_2)\} \subseteq W$,  the initial strong components $K_1$ and $K_2$ must be adjacent which is a contradiction. 
\qed
\end{proof}

For next lemma, we need the following auxiliary result. 

\begin{lemma}[Wang e Wang~\cite{wang2011independent}, 2011]
\label{arc-in-lema-inicial}
Let $D$ be a connected non-strong arc-locally in-semicomplete digraph. If there are more than one initial strong component, then all initial strong components are trivial.
\end{lemma}

\begin{lemma}
\label{lem-arc-v1-v2}
Let $Q$ be a strong component that induces an odd extended cycle of length at least five of an arc-locally in-semicomplete digraph $D$. If $Q$ is an initial strong component of $D$, then $V(D)$ admits a partition $(V_1, V_2)$, such that $V_1 \Rightarrow V_2$, $D[V_1] = V(Q)$, $D[V_2]$ is a bipartite digraph and $V_2$ can be empty.
\end{lemma}

\begin{proof}

Let $Q:=Q[X_1,X_2,...,X_{2k+1}]$. If $V(D)=V(Q)$, then the result follows by taking the partition $(V(Q), \emptyset)$. So we may assume that $D-V(Q)$ is nonempty. In particular, $D$ is non-strong. By Lemma~\ref{arc-in-lema-inicial}, $Q$ is the only initial strong component of $D$. Consider the set $V_2 = \cup_ {K \in \K^+(Q)} V(K)$. Note that $V(D) = V(Q) \cup V_2$ and $V(Q) \Rightarrow V_2$. Now, we show that $D[V_2]$ is a bipartite digraph. By Lemma~\ref{lem-arc-cir-est}(i), every vertex of $V_2$ is a trivial strong component and hence, $D[V_2]$ is an acyclic digraph. Recall that since $D$ is arc-locally in-semicomplete, if $u_1,u_2,u_3$ and $u_4$ are distinct vertices such that $u_2$ and $u_3$ are adjacent, $u_1 \to u_2$ and $u_4 \to u_3$, then $u_1$ and $u_4$ must be adjacent in $D$.

\textbf{Claim 1.} If a vertex $u$ dominates a vertex $v$ of a transitive triangle $T$, then $u$ is adjacent to a vertex $w$ distinct of $v$ in $V(T)$ such that $D[\{u, v, w \}]$ is a transitive triangle. In particular, if $u \in V(Q)$, then $u$ dominates both $v$ and $w$, because $V(Q) \Rightarrow V_2$. In fact, let $V(T) = \{y_1, y_2, y_3\}$. Assume that $u \to y_1$. If $y_2 \to y_3$ (resp., $y_3 \to y_2$), then $u$ and $y_2$ (resp., $y_3$) are adjacent. Therefore, $D[\{u, y_1, y_2 \}]$ (resp., $D[\{u, y_1, y_3\}]$) is a transitive triangle, since $D[V_2]$ is an acyclic digraph. This proves the claim.

\textbf{Claim 2.} There are no index $i \in \{1,2,\ldots,2k+1\}$, vertices $x_{i-1} \in X_{i-1}$, $x_i \in X_i$ and $y \in V_2$ such that $D[\{x_{i-1},x_i,y\}]$ is a transitive triangle (thus,  $x_{i-1} \to \{x_i,y\}$ and $x_i \to y\}$). In fact, let $x_{i-2} \in X_{i-2}$. Since $x_{i-2} \to x_{i-1}$, $x_i \to y$ and $x_{i-1}y$ is an edge, we conclude that $x_{i-2}$ and $x_i$ must be adjacent, a contradiction. This proves the claim.

By Lemma~\ref{arc-in-circuito}, it follows that $D$(and hence, $D[V_2]$) contains no induced non-oriented odd cycle of length at least five. Therefore, it suffices to show that $D[V_2]$ is transitive triangle free. Towards a contradiction, assume that $D[V_2]$ contains a transitive triangle. Let $T$ be a transitive triangle of $D[V_2]$ such that $dist(Q,T)$ is minimum. Assume that $V(T) = \{y_1, y_2, y_3\}$. Let $P = w_1w_2w_l ... w_{l + 1}$ be a minimum path from $Q$ to $T$. Without loss of generality, assume that $w_{l + 1} = y_1$. First, assume that $l>1$. By Claim 1, the digraph $D[\{w_l, y_1, y_j\}]$ is a transitive triangle, for some $j \in \{2,3\}$, which contradicts the choice of $T$. Thus, it follows that $l = 1$, that is, there is an edge from $Q$ to $T$. Let $x_i \in X_i$ be a vertex of $Q$ that dominates $y_1$ in $V(T)$. By Claim 1, it follows that $D[\{x_i, y_1, y_j\}]$ is a transitive triangle, for some $j \in \{2,3\}$. Let $x_{i-1} \in X_{i-1}$. By definition of extended cycle, $x_{i-1} \to x_i$. If $y_1 \to y_j$ (resp., $y_j \to y_1$), then since $x_i \to \{y_1,y_j\}$ and $x_{i-1} \to x_i$, it follows that $x_{i-1} \to y_1$ (resp., $x_{i-1} \to y_j$), a contradiction by Claim 2. Therefore, $D[V_2]$ is a bipartite digraph.

Since $Q$ is the only initial strong component of $D$ and $V(D) = V(Q) \cup V_2$, it follows that $V(D)$ can be partition into $(V(Q), V_2)$ such that $V(Q) \Rightarrow V_2$ and $D[V_2]$ is a bipartite digraph. This ends the proof.
\qed 
\end{proof}

The next lemma states that if an arc-locally in-semicomplete digraph $D$ contains a non-initial strong component $Q$ that induces an odd extended cycle of length at least five, then $V(D)$ admits a similar partition to the previous lemma or a clique cut.

\begin{lemma}
\label{lem-arc-digrafo-D}
Let $D$ be a connected non-strong arc-locally in-semicomplete digraph and let $Q$ be a non-initial strong component of $D$ that induces an odd extended cycle of length at least five. Then, $D$ has a clique cut or $V(D)$ admits a partition $(V_1, V(Q), V_3)$, such that $D[V_1]$ is a semicomplete digraph, $V_1 \mapsto V(Q)$, $V_1 \Rightarrow V_3$, $V(Q) \Rightarrow V_3$ and $D[V_3]$ is a bipartite digraph($V_3$ could be empty).
\end{lemma}

\begin{proof}
Consider the sets $V_1 = \cup_{K \in \K^-(Q)}V(K)$ and $V_3 = \cup_{K \in \K^+(Q)}V(K)$. Note that only $V_3$ can be empty. By Lemma~\ref{lem-arc-cir-est}(iii), it follows that $D[V_1]$ is a semicomplete digraph. By Lemma~\ref{lem-arc-cir-est}(iv), there exists only one initial strong component $K$ that dominates $Q$ in $D$. Note that $V(K) \subseteq V_1$. Consider the vertex set $B = \{V_1 \cup V(Q) \cup V_3\}$. We split the proof in two cases, depending on whether $V(D)= B$ or not.

\textbf{Case 1.} Assume that $V(D)=B$. Consider the digraph $H=D-V_1$. Note that $V(H)= V(Q) \cup V_3$ and $Q$ is the unique initial strong component of $H$. By Lemma~\ref{lem-arc-v1-v2} applied to $H$, it follows that $V(Q) \Rightarrow V_3$ and $D[V_3]$ is a bipartite digraph. By Lemma~\ref{lem-arc-cir-est}(ii), it follows that $V_1 \mapsto V(Q)$. By definition of $\K^-(Q)$ and $\K^+(Q)$, we conclude that $V_1 \Rightarrow V_3$. Therefore, $(V_1, V(Q), V_3)$ is the desired partition of $V(D)$.

\textbf{Case 2.} Assume that $B$ is a proper subset of $V(D)$. In this case, we show that $V_1$ is a clique cut of $D$. First, we show that there exists no vertex $v$ in $V(D)-B$ adjacent to $V(Q) \cup V_3$. Since $v \notin B$, vertex $v$ does not dominate and it is not dominated by any vertex in $V(Q)$, nor it is dominated by any vertex in $V_3$. Thus, it suffices to show that $v$ does not dominate any vertex of $V_3$. Towards a contradiction, assume that there exists $v \in V(D)-B$ that dominates a vertex $u$ of $V_3$. Choose $u$ such that $dist(Q,u)$ is minimum. Let $P = w_1w_2w_l \ldots u $ be a minimum path from $Q$ to $u$. Assume that $l>1$. Note that $w_1 \in V(Q)$ and $w_2,...,w_l, u \in V_3$. Since $v \to u$, $w_l \to u$, $w_ {l-1} \to w_l$ and $D$ is arc-locally in-semicomplete, it follows that $v \to w_{l-1}$, which contradicts the choice of $u$. Thus, we may assume that $w_1 \to u$ and $ v \to u$. Let $z$ be a vertex of $V(Q)$ that dominates $w_1$. Since $D$ is arc-locally in-semicomplete, $w_1 \to u$, $v \to u$ and $z \to w_1$, it follows that $z$ and $v$ are adjacent, a contradiction to the fact that $v \notin B$. Since $D$ is connected, $B$ is a proper subset of $D$, $D[V_1]$ is a semicomplete digraph and there exists no vertex in $V(D)-B$ adjacent to $V(Q) \cup V_3$, we conclude that $V_1$ is a clique cut. This finishes the proof.
\qed
\end{proof}

For the main result of this section, we need the following auxiliary result.

\begin{lemma}[Wang and Wang~\cite{wang2009structure}, 2009]
\label{arc-cycle-extended-cycle}
Let $D$ be a strong arc-locally in-semicomplete digraph. If $D$ contains an induced cycle of length at least five, then $D$ is an extended cycle.
\end{lemma}

\begin{theorem}
\label{lem-arc-resultado}
Let $D$ be a connected arc-locally in-semicomplete digraph. Then,
\begin{enumerate}
	\item[\rm{i}.] $D$ is a diperfect digraph, or

	\item[\rm{ii}.] $V(D)$ can be partitioned into $(V_1, V_2, V_3)$ such that $D[V_1]$ is a semicomplete digraph, $V_1 \mapsto V_2$, $V_1 \Rightarrow V_3$, $D[V_2]$ is an odd extended cycle of length at least five, $V_2 \Rightarrow V_3$, $D[V_3]$ is a bipartite digraph and $V_1$ or $V_3$ (or both) can be empty, or 
	
	\item[\rm{iii}.] $D$ has a clique cut.
\end{enumerate} 
\end{theorem}

\begin{proof}
If $D$ contains no induced odd cycle of length at least five, then by Lemma~\ref{arc-no-cc-perf} the digraph $D$ is diperfect. Thus, let $C$ be an induced odd cycle of length at least five. Let $Q$ be the strong component that contains $C$. By Lemma~\ref{arc-cycle-extended-cycle}, the strong component $Q$ induces an odd extended cycle of length at least five. Now, we have two cases to deal with, depending on whether $Q$ is an initial strong component or not. First, assume that $Q$ is an initial strong component of $D$. By Lemma~\ref {lem-arc-v1-v2}, the set $V(D)$ admits a partition $(V_1, V(Q), V_3)$ such that $V_1$ is empty, $V(Q) \Rightarrow V_3$ and $D[V_3]$ is a bipartite digraph, and so (ii) holds. Now, assume that $Q$ is not an initial strong component of $D$. By Lemma~\ref{lem-arc-digrafo-D}, it follows that $D$ contains a clique cut and hence (iii) holds, or $V(D)$ can be partitioned into $(V_1, V(Q), V_3)$ such that $D[V_1]$ is a semicomplete digraph, $V_1 \mapsto V(Q)$, $V_1 \Rightarrow V_3$, $V(Q) \Rightarrow V_3$ and $D[V_3]$ is a bipartite digraph and hence (ii) holds. This ends the proof.
\qed
\end{proof}

Note that the inverse of an arc-locally in-semicomplete digraph is an arc-locally out-semicomplete digraph. Thus, we have the following result.

\begin{theorem}
\label{lem-arc-resultado}
Let $D$ be a connected arc-locally out-semicomplete digraph. Then,
\begin{enumerate}
	\item[\rm{i}.] $D$ is a diperfect digraph, or

	\item[\rm{ii}.] $V(D)$ can be partitioned into $(V_1, V_2, V_3)$ such that $D[V_1]$ is a semicomplete digraph, $V_2 \mapsto V_1$, $V_3 \Rightarrow V_1$, $D[V_2]$ is an odd extended cycle of length at least five, $V_3 \Rightarrow V_2$, $D[V_3]$ is a bipartite digraph and $V_1$ or $V_3$ (or both) can be empty, or 
	
	\item[\rm{iii}.] $D$ has a clique cut.
\qed
\end{enumerate} 
\end{theorem}

\section{Arc-locally semicomplete digraphs}
\label{arc-semi}

Recall that, a digraph $D$ is arc-locally semicomplete if $D$ is both arc-locally in-semicomplete and arc-locally out-semicomplete. In~\cite{galeana12}, Galeana-Sánchez and Goldfeder presented a characterization of arbitrary connected arc-locally semicomplete digraphs. In order to present the result obtained by Galeana-Sánchez and Goldfeder, we need some definitions. 

Let $D$ be a digraph with $V(D) = \{v_0 , \ldots , v_n\}$ and let $H_0 , \ldots , H_n$ digraphs indexed by $V(D)$. The \emph{composition} $D[H_0 , \ldots , H_n]$ is the digraph $H$ with vertex set $V(H) = \cup_{i=0}^n V(H_i)$ and, $E(H)= \cup_{i=0}^n E(H_i) \cup \{uv: u \in V(H_i), v \in V(H_j) $ and $ v_iv_j \in E(D) \}$. If each $V(H_i)$ is a stable set, then we call $H$ an \emph{extension} of $D$. Let $X$ and $Y$ be two disjoint subsets of vertices of $D$. If $\sP = ( V_0 , \ldots, V_{k-1})$ is a fixed ordered $k$-partition of $V(D)$, we say that $X$ $\sP$-dominates $Y$ according to the ordered $k$-partition $\sP$ , $X \to^{\sP} Y$, if for each vertex $u$ in $X \cap V_k$ and each vertex $v$ in $Y \cap V_l$ , we have $u \to v$ whenever $k \neq l$ and there is no edge from $Y$ to $X$. Let $D_0 , \ldots , D_n$ be $k$-partite digraphs with fixed ordered $k$-partitions $\sP(D_i) = ( V_0^i , \ldots , V_{k-1}^i)$, the $\sP$-composition according to the ordered $k$-partition $\sP = (\cup_{i=0}^{n}(V_0^i \times \{ i \}) = V_0 , \ldots, \cup_{i=0}^{n}(V_{k-1}^i \times \{ i \}) = V_{k-1})$, denoted by $D[D_0, \ldots , D_n ]^{\sP}$ , is the digraph $H$ with vertex set $V(H) = \cup_{i=0}^{k-1}V i$ and, for $(w,i),(z,j)$ in $V(H)$ , the edge $(w,i) \to (z,j)$ is in $E(H)$ if $i = j$ and $w \to z$ in $D_i$ or $w \in V_k^i$, $z \in V_l^j$ with $k \neq l$ and $v_iv_j$ in $E(D)$. Moreover, $\sC$ denotes the digraph with vertex set $\{v_1, v_2 , v_3\}$ and edge set $\{v_1v_2, v_2v_3,v_3v_1\}$, $TT_3$ denotes the digraph with vertex set $\{ u_1, u_2 , u_3 \}$ and edge set $\{u_1u_2, u_2u_3, u_1u_3\}$ and we denote by $E_m$ the digraph with $m$ vertices and no edges.

\begin{theorem}[Galeana-Sánchez and Goldfeder~\cite{galeana12}, 2016]
Let $D$ be a connected digraph. Then, $D$ is an arc-locally semicomplete digraph if and only if $D$ is one of the following:
\begin{enumerate}
	\item a digraph with at most three vertices,
	\item a subdigraph of an extension of one edge,
	\item a semicomplete bipartite subdigraph of $\overrightarrow{P_2}[E_{m_0}, \overrightarrow{C_2}[ E_{m_1},E_{m_2}], E_{m_3}]^{\sP}$, $m_1 = 1$ and if $m_2 > 1$, then the partition is $\sP = ( E_{m_1} , E_{m_0} \cup E_{m_2} \cup E_{m_3})$,
	\item $\sC^{*}_3[E_1, E_n, E_1]$,
	\item $TT_3[E_1,E_n,E_1]$ ,
	\item an extension of a directed path or an extension of a directed cycle,
	\item $\overrightarrow{P_2}[ E_{m_0}, D', E_{m_2}]^{\sP} \leq D \leq TT_3 [ E_{m_0},D', E_{m_2}]^{\sP}$, where $D'$ is a semicomplete bipartite digraph (it could have no edges),
	\item $\overrightarrow{P_2}[E_1, D', E_1]$ , where $D'$ is a semicomplete digraph,	
	\item a semicomplete bipartite digraph, or
	\item a semicomplete digraph.
\end{enumerate}
\end{theorem}

In this paper, we present another characterization for this class. We show that a connected arc-locally semicomplete digraph is either diperfect or an odd extended cycle of length at least five. Note that the inverse of an arc-locally semicomplete digraph is also an arc-locally semicomplete digraph. This principle of directional duality is very useful to fix an orientation for a given path or edge in a proof. Besides, every result valid for arc-locally in-semicomplete digraphs, also holds for arc-locally semicomplete digraphs, because they form a subclass of the former one.

The next lemma states if a connected arc-locally semicomplete digraph $D$ contains an induced odd extended cycle $Q$ of length at least five, then $V(D)=V(Q)$. 

\begin{lemma}
\label{lem_arc_semi}
Let $D$ be an connected arc-locally semicomplete digraph. If $D$ contains a strong component $Q$ that induces an odd extended cycle of length of at least five, then $V(D)=V(Q)$.
\end{lemma}

\begin{proof}
Let $Q: = Q[X_1, X_2, \ldots, X_{2k + 1}]$ be a strong component $Q$ that induces an odd extended cycle of length at least five. We show that $V(Q) = V(D)$. Assume, without loss of generality, that there exists a vertex $u \in V(D)-V(Q)$ such that $u$ dominates some vertex of $Q$. Since $D[V(Q)]$ is a non-bipartite digraph, it follows by Lemma~\ref{arc-in-lema_1} that $u \mapsto Q$. Consider vertices $x_1 \in X_1 $, $ x_2 \in X_2 $ and $ x_3 \in X_3$. Since $u \to \{x_1, x_2 \}$, $ x_2 \to x_3 $ and $D$ is arc-locally semicomplete, it follows that $x_1$ and $x_3$ are adjacent, a contradiction to the fact that $Q$ induces an extended cycle. Since $D$ is connected, then $V(D)=V(Q)$. 
\qed
\end{proof}

Now, we are ready for the main result of this section.

\begin{theorem}
\label{cara_arc_semi}
Let $D$ be a connected arc-locally semicomplete digraph. Then, $D$ is either a diperfect digraph or an odd extended cycle of length at least five. 
\end{theorem}
\begin{proof}
Let $D$ be a connected arc-locally semicomplete digraph. If $D$ contains no induced odd cycle of length at least five, then by Lemma~\ref{arc-no-cc-perf} the digraph $D$ is diperfect. Thus, let $C$ be an induced odd cycle of length at least five. Let $Q$ be the strong component that contains $C$. By Lemma~\ref{arc-cycle-extended-cycle}, the strong component $Q$ induces an odd extended cycle of length of at least five. Then, by Lemma~\ref{lem_arc_semi} we conclude that $V(D)=V(Q)$. 
\qed
\end{proof}

\bibliographystyle{cas-model2-names}


\bibliography{paper}

\end{document}